\documentclass{amsart}
\usepackage{amsmath}
\usepackage{amsfonts}
\usepackage{graphics}
\usepackage{epsfig}
\usepackage{amssymb}
\usepackage{amscd}
\usepackage[all]{xy}
\usepackage{latexsym}
\usepackage{graphicx}
\usepackage{multirow}
\usepackage{geometry}

\theoremstyle{remark}
\newtheorem{lem}{\bf Lemma}[section]

\newtheorem{thm}{\bf Theorem}[section]
\newtheorem{cor}{\bf Corollary}[section]

\newtheorem{prop}{\bf Proposition}[section]

\numberwithin{equation}{section} \numberwithin{figure}{section}

\renewcommand*{\to}{\rightarrow}
\renewcommand*{\bar}[1]{\overline{#1}}

\newcommand{\ch}{\operatorname{ch}}

\newcommand{\Gr}{\operatorname{Gr}}

\newcommand{\mb}[1]{\mathbb{#1}} 

\newcommand{\hs}{\mathcal{H}}
\newcommand{\mc}[1]{\mathcal{#1}}

\begin{document}
\large \setcounter{section}{0}

\title{Equivariant cohomology of infinite-dimensional Grassmannian and shifted Schur functions}

\author{Jia-Ming (Frank) Liou, Albert Schwarz}

\begin{abstract}\large
We  study the multiplication and comultiplication in equivariant cohomology of Sato Grassmannian  .
\end{abstract}

\maketitle

\allowdisplaybreaks



\allowdisplaybreaks

\section {Introduction}
Let us consider the Hilbert space $\hs=L^{2}(S^{1})$ and its
subspaces $\hs _+$, $\hs_{-}$ defined as closed subspaces of $\hs$
spanned by $\{z^{i}:i\geq 0\}$ and $\{z^{j}:j<0\}$ respectively .
Using the decomposition $\hs=\hs _+\oplus \hs_-$ one can define the
(Segal-Wilson version of) Sato Grassmannian $\Gr (\hs)$ as a space
of all closed linear subspaces $W\subset \hs$ such that the
projection $\pi _-:W\to \hs _-$ is a Fredholm operator and the
projection $\pi_+:W\to \hs _+$ is a compact operator (see \cite{PS}
for more detail).

The group $S^1$ acts naturally on $\hs$: to every $\alpha $ obeying
$|\alpha|=1$ we assign a map $f(z)\to f(\alpha z)$. This action (we
will call it standard action) generates an action of $S^1$ on
Grassmannian . Representing a function on a circle as a Fourier
series $f(z)=\sum a_nz^n$, we see that the standard action sends
$a_k\to \alpha ^ka_k.$  The standard action is related to the
$S^1$-action on various moduli spaces embedded into Grassmannian by
means of Krichever construction, see \cite {Sch} and \cite{LS} for
more detail. Motivated by the desire to construct
nonperturbative string theory we considered in \cite {LS} equvariant cohomology of
Sato Grassmannian and homomorphism of this cohomology induced by the
Krichever map. The results of present paper will be used in \cite
{LSw} to study this homomorphism in more detail.

One can consider  more general actions of $S^1$ on $\hs$ sending
$a_k\to \alpha ^{n_k}a_k$ where ${n_k}\in \mathbb{Z}$ is an
arbitrary doubly infinite sequence of integers. This action also
generates an action of $S^1$ on Grassmannian. The Grassmannian $\Gr
(\hs)$ is a disjoint union of connected components $\Gr _d(\hs)$
labeled by the index of the projection $\pi _-:W\to \hs _-$ . All
components are homeomorphic. It was proven in \cite {PS} that every
component is homotopy equivalent to a subspace having a cell
decomposition $K=\cup \sigma _{\lambda}$ consisting of
even-dimensional cells (finite-dimensional Schubert cells). The
cells are labeled by partitions. This decomposition is
$S^1$-invariant with respect to any action of $S^1$ from the class
of actions we are interested in. This allows us to say that the
equivariant cohomology $H_{S^1}(\Gr_d(\hs))$ has a free system of
generators $\Omega_{\lambda}^{T}=[\bar{\Sigma} _{\lambda}]$ as a
module over $H_{S^1}(\mbox{pt}).$ These generators can be
interpreted also as cohomology classes dual to Schubert cycles
$\bar{\Sigma} _{\lambda}$ having finite codimension, In this letter
we calculate the multiplication table in the basis
$\Omega_{\lambda}^{T}$.

\begin{thm}\label{coeff0}
For the standard $S^1$-action on $\Gr_{d}(\hs)$ the coefficients in
the decomposition
\begin{equation}\label{coef0}
\Omega_{\lambda}^{T}\Omega_{\mu}^{T}=\sum_{\nu}C_{\lambda
\mu}^{\nu}(u)\Omega_{\nu}^{T}
\end{equation}
can be expressed in terms of coefficients in the decomposition
\begin{equation}\label{coef1}
s^* _{\lambda}s^*_{\mu}=\sum_{\nu}C_{\lambda \mu}^{\nu}s^* _{\nu}
\end{equation}
by the formula $C_{\lambda \mu}^{\nu}(u)=C_{\lambda
\mu}^{\nu}u^{|\lambda|+|\mu|-|\nu|}$. Here $s^*_{\lambda}$ stands
for the shifted Schur function in the sense  of \cite {OO} and
$|\lambda|=\sum_{i}\lambda_{i}$ is the weight of the partition
$\lambda=(\lambda_{i})$.
\end{thm}

Combining this theorem with  the results of \cite {Mol} we obtain
the following expression for the coefficients.
\begin{cor}
\begin{equation*}
C_{\lambda\mu}^{\nu}(u)=u^{|\lambda|+|\mu|-|\nu|}\sum_{\lambda,\mu\subset
\rho,\nu}(-1)^{|\nu|-|\rho|}\frac{h(\rho)}{h(\nu/\rho)h(\rho/\lambda)h(\rho/\mu)}.
\end{equation*}
Here the function $h$ is defined to be
$h(\nu/\mu)=|\nu/\mu|!/\dim\left(\nu/\mu\right)$ for any skew
diagram $\nu/\mu$ and $\dim(\nu/\mu)$ is the number of standard
$\nu/\mu$-tableaux.
\end{cor}

The Theorem \ref {coeff0} can be applied to the analysis of
multiplication in tautological cohomology ring of the  moduli spaces
of complex curves (more precisely, to the study of intersection of
cycles defined in terms of Weierstrass points), see \cite {LSw}. To
generalize the theorem to the non-standard action of $S^1$ we
introduce the notion of shifted double Schur function (it is closely
related to double Schur functions  of infinite number of arguments
introduced in \cite {Mol2}).

Let $x=(x_{1},\cdots,x_{n})$ be an $n$-tuple of variables and
$y=(y_{i})_{i\in\mb Z}$ be a doubly infinite sequence. Recall that,
the double Schur function $^n s_{\lambda}(x_{1},\cdots,x_{n}|y)$ is
{\footnote { We add the index $n$ to the conventional notation
$s_{\lambda}$ to emphasize that the function depends on $n$
variables $x_k$.}} a symmetric polynomial in
$x=(x_{1},\cdots,x_{n})$ with coefficients in $\mb C[y]$ defined by
\begin{equation*}
^{n}s_{\lambda}(x_{1},\cdots,x_{n}|y)=\det\left[(x_{i}|y)^{\lambda_{j}+n-j}\right]/\det\left[(x_{i}|y)^{n-j}\right],
\end{equation*}
where $(x_{i}|y)^{p}=\prod_{j=1}^{p}(x_{i}-y_{j})$. Let us introduce
the shift operator on $\mb C[y]$ given by
\begin{equation*}
(\tau y)_{i}=y_{i-1},\quad i\in\mb Z.
\end{equation*}
Then the double Schur function satisfies the generalized
Jacobi-Trudi formula \cite{Ful}:
\begin{equation}\label{GJT}
^{n}s_{\lambda}(x_{1},\cdots,x_{n}|y)=\det\left[h_{\lambda_{i}+j-i}(x_{1},\cdots,x_{n}|\tau^{j-1}y)\right]_{i,j=1}^{n},
\end{equation}
where
\begin{equation*}
h_{p}(x_{1},\cdots,x_{n}|y)=\sum_{1\leq i_{1}\leq \cdots\leq
i_{p}\leq k}(x_{i_{1}}-y_{i_{1}})\cdots
(x_{i_{p}}-y_{i_{p}+p-1}),\quad p\geq 1.
\end{equation*}
The Littlewood-Richardson coefficients $^{n}c_{\lambda\mu}^{\nu}(y)$
of the double Schur functions are defined by
\begin{equation}\label{coeffl}
^{n}s_{\lambda}(x_{1},\cdots,x_{n}|y)\
^{n}s_{\mu}(x_{1},\cdots,x_{n}|y)=\sum_{\nu}\
^{n}c_{\lambda\mu}^{\nu}(y)\ ^{n}s_{\mu}(x_{1},\cdots,x_{n}|y).
\end{equation}
These coefficients $^{n}c_{\lambda\mu}^{\nu}(y)$ were calculated in
\cite{Mol}. We define the shifted double Schur function by
\begin{equation}\label{sds1}
^ns_{\lambda}^{*}(x_1,\cdots,x_n|y)=^ns_{\lambda}(x_{1}+y_{-1},x_{2}+y_{-2},\cdots,x_{n}+y_{-n}|\tau^{n+1}y).
\end{equation}
Under the change of variables $x_{i}'=x_{i}+y_{-i}$ for $1\leq i\leq
n$, the shifted double Schur function
$^ns_{\lambda}^{*}(x_{1},\cdots,x_{n}|y)$ becomes the double Schur
function $^ns_{\lambda}(x_{1}',\cdots,x_{n}'|\tau^{n+1}y)$. Notice
that the shifted double Schur function
$^ns_{\lambda}^{*}(x_{1},\cdots,x_{n}|y)$ is the shifted Schur
function defined in \cite{OO} if $y=(y_{k})_{k}$ is the sequence
defined by the relation $y_{k}=\mbox{constant}+k$ for all $k$.
The shifted double Schur functions
$^{n}s_{\lambda}^{*}(x_{1},\cdots,x_{n}|y)$ have the following
stability property:
\begin{prop}\label{stab}
If $l(\lambda)<n$, then
\begin{equation*}
^{n+1}s_{\lambda}^{*}(x_{1},\cdots,x_{n},0|y)=^{n}s_{\lambda}^{*}(x_{1},\cdots,x_{n}|y).
\end{equation*}
Here $l(\lambda)$ is the length of a partition $\lambda$.
\end{prop}
Using this property, we can define the shifted double Schur function
$s_{\lambda}^*(x|y)$ depending on infinite number of arguments
$x=(x_i)_{i\in\mb N}$ and $y=(y_j)_{j\in\mb Z}$  (we assume that only finite number of variables $x_i$
does not vanish). Namely, we define
\begin{equation}\label{sds}
s_{\lambda}^{*}(x|y)=\ ^{n}s_{\lambda}^{*}(x_{1},\cdots,x_{n}|y)
\end{equation}
where $n$ is chosen in such a way that $n>l(\lambda)$ and $x_i=0$
for $i>n$. The coefficients in the decomposition
\begin{equation}\label{coeff1}
s_{\lambda}^{*}(x|y)s_{\mu}^{*}(x|y)=\sum
C_{\lambda\mu}^{\nu}(y)s_{\nu}^{*}(x|y).
\end{equation}
can be obtained from (\ref{coeffl}) and from the results of \cite
{Mol}, \cite {Mol2} or \cite {KT}. In fact, the relation between
$C_{\lambda\mu}^{\nu}$ and $^{n}c_{\lambda\mu}^{\nu}$ is given by
\begin{equation*}
C_{\lambda\mu}^{\nu}(y)=\ ^{n}c_{\lambda\mu}^{\nu}(\tau^{n+1}y)
\end{equation*}
for $n>l(\lambda),l(\mu),l(\nu)$.

The shifted double Schur functions belong to the ring $\Lambda^{*}
(x\|y)$ that consists of polynomial functions that depend on sequences
$x=(x_{n})_{n\in \mb N}$ and $y=(y_{i})_{i\in\mb Z}$ and are symmetric
with respect to shifted variables $x_{i}'=x_{i}+y_{-i}$ (we assume
that $x_n=0$ for $n>>0$). Moreover, they form a free system of
generators of $\Lambda^{*} (x\|y)$ considered as $\mb C[y]$-module.
Notice that the ring $\Lambda^{*}(x\|y)$ is obviously isomorphic to
the ring $\Lambda(x\|y)$ constructed in \cite {Mol2}. The following
theorem describes the multiplication in the equivariant cohomology
$H_{S^1}(\Gr_d(\hs))$ for non-standard action of $S^1$:

\begin {thm}\label{t2}
\begin{equation*}
\Omega_{\lambda}^{T}\Omega_{\mu}^{T}=\sum_{\nu}C_{\lambda\mu}^{\nu}(n)\Omega_{\nu}^{T},
\end{equation*}
where $C_{\lambda\mu}^{\nu}(n)$ are the coefficients in (\ref
{coeff1}) calculated for $y_k=n_{k+d}u$ and $n_k$ denotes the
sequence specifying the action of $S^1.$
\end {thm}

The Theorem \ref{coeff0} is a particular case of Theorem \ref {t2}
for $n_k=k.$\\

Let us consider the infinite-dimensional torus  $\mb T$ and its
action on the Grassmannian.  Algebraically the infinite torus $\mb
T$ is the infinite direct product $\prod_{i\in\mb Z}S^{1}$. The
action of $\mb T$ on Grassmannian corresponds to the action on $\hs$
transforming $a_{k}\to \alpha_{k}a_{k}$, where $(\alpha_{k})\in\mb
T$ and $f=\sum_{n}a_{n}z^{n}\in\hs$. This action specifies an
embedding of $\mb T$ into the group of unitary transformations of
$\hs$; the topology of $\mb T$ is induced by this embedding.

One can prove that the equivariant cohomology $H_{\mb T}(pt)$
(cohomology of the classifying space $B_{\mb T}$) is isomorphic to
the polynomial ring $\mb C [{\bf u}]$ where $\bf u$ stands for the
doubly infinite sequence $u_k$. The proof is based on the
consideration of homomorphisms of $\mb T$ onto finite-dimensional
tori and homomorphisms of $S^1$ into $\mb T.$ The finite
codimensional Schubert cells are $\mb T$-invariant and hence the
equivariant cohomology $H_{\mb T}(\Gr_d (\hs))$ is a free $H_{\mb
T}(pt)$ -module generated by cohomology classes labeled by
partitions; we denote these classes by $\Omega_{\lambda}^{T}.$

The submanifold $\Gr_{d}^{l}$ of $\Gr_{d}(\hs)$ consisting of points
$W$ so that the orthogonal projection $\pi_{l}:W\to z^{-l}\hs_{-}$
is surjective. The intersection of Schubert cycle
$\bar{\Sigma}_{\lambda}$ and $\Gr_{d}^{l}$ is denoted by
$\bar{\Sigma}_{\lambda,l}$. The  equivariant cohomology class   in
$H_{\mb T}^{*}(\Gr_{d}^{l})$ corresponding to
$\bar{\Sigma}_{\lambda,l}$ is denoted by $\Omega_{\lambda,l}^{T}$ .
The Schubert cycle $\bar{\Sigma}_{\lambda}$ and $\Gr_{d}^{l}$ are in
general position if $l(\lambda)<d+l$; then we have
\begin{equation*}
f_{l}^{*}\Omega_{\lambda}^{T}=\Omega_{\lambda,l}^{T},
\end{equation*}
where $f_{l}^{*}:H_{\mb T}^{*}(\Gr_{d}(\hs))\to H_{\mb
T}^{*}(\Gr_{d}^{l})$ is map induced by the inclusion map
$f_{l}:\Gr_{d}^{l}\to \Gr_{d}(\hs)$.

The classes  $\Omega_{\lambda,l}^{T}$ for $l(\lambda)<d+l$ form an
additive system of generators of equivariant cohomology. It follows
from (\ref{coeffl}) that
\begin{thm}
The multiplication table in $H_{\mb T}^{*}(\Gr_{d}^{l})$ is given by
\begin{equation*}
\Omega_{\lambda,l}^{T}\Omega_{\mu,l}^{T}=\sum_{\nu}\
^{d+l}c_{\lambda\mu}^{\nu}(y)\Omega_{\nu,l}^{T},
\end{equation*}
where $y=(y_{i})_{i\in\mb Z}$ is the sequence defined by
$y=\tau^{l+1}\mathbf{u}$.
\end{thm}

Using the coefficients in (\ref{coeff1}), we can calculate
the multiplication table in equivariant cohomology of Grassmannian
$\Gr_{d}(\hs)$ with respect to the infinite torus action. Namely,

\begin {thm}\label{coeff}
The multiplication table in $H_{\mb T}^{*}(\Gr_{d}(\hs))$ is given
by:
\begin{equation*}
\Omega_{\lambda}^{T}\Omega_{\mu}^{T}=\sum_{\nu}C_{\lambda\mu}^{\nu}(y)\Omega_{\nu}^{T},
\end{equation*}
where $(y_{k})$ is the sequence given by the relation
$y=\tau^{-d}\mathbf{u}$. The coefficient in this equation comes from
the formula (\ref {coeff1}); it is considered as an element of
$H_{\mb T}(pt)$ .
 \end {thm}

Using homomorphisms of $S^1$ into $\mb T$ one can derive this
theorem from Theorem \ref {t2} (conversely, one can deduce Theorem
\ref {t2} from Theorem \ref {coeff}). It follows from Theorem \ref
{coeff} that the equivariant cohomology $H_{\mb
T}^{*}(\Gr_{d}(\hs))$ is isomorphic to the ring $\Lambda^{*}
(x\|y)$.

We can introduce a comultiplication in the ring $H_{\mb
T}^{*}(\Gr_{d}(\hs))$ using the map
\begin{equation}\label{product}
\rho:\Gr_0(\hs ')\times\Gr_0(\hs '')\to \Gr_0(\hs'\oplus \hs'')
\end{equation}
defined by $(V,W)\mapsto V\oplus W$. Namely, we take $\hs '=\hs
_{even}$ (the subspace spanned by $z^{2k}$) and $\hs ''=\hs _{odd}$
(the space spanned by $z^{2k+1}$). Then $\hs'\oplus \hs ''=\hs$ and
the map $\rho$ determines a homomorphism of $\mb T$- equivariant
cohomology
\begin{equation*}
H_{\mb T}(\Gr_0(\hs))\to H_{\mb T}(\Gr_0(\hs _{even})\times\Gr_0(\hs
_{odd})).
\end{equation*}
It is easy to prove that $H_{\mb T}(\Gr_0(\hs _{even})\times\Gr_0(\hs
_{odd}))$ is isomorphic to tensor product of two copies of  $H_{\mb
T}^{*}(\Gr_0(\hs))$; we obtain a comultiplication $\Delta$ in $H_{\mb
T}^{*}(\Gr_0(\hs)).$

Let us introduce the notion of the $k$-th shifted power sum function
in $\Lambda^{*}(x\|y)$:
\begin{equation}\label{spower}
p_{k}(x|y)=\sum_{i=1}^{\infty}\left[(x_{i}+y_{-i})^{k}-y_{-i}^{k}\right].
\end{equation}
Here we assume that $x_{n}=0$ for $n\gg 0$; hence (\ref{spower}) is
a finite sum. We will prove that
\begin{equation}\label{comp}
\Delta p_{k}=p_{k}\otimes 1+1\otimes p_{k},\quad k\geq 1.
\end{equation}
In other words, we have the following theorem:

\begin {thm}\label{com}
The comultiplication $\Delta$ in the ring $H_{\mb
T}^{*}(\Gr_{d}(\hs))$ defined above coincides with the
comultiplication in the ring $\Lambda^{*} (x\|y)=\Lambda(x\|y)$
constructed in \cite {Mol2}.
\end {thm}

\section{Equivariant Schubert classes}

The Grassmannian $\Gr(\hs)$ has a stratification in terms of
Schubert cells having finite codimension; it is a disjoint union of
$\mb T$-invariant submanifolds $\Sigma_{S}$ labeled by $S$, where
$S$ is a subset of $\mb Z$ such that the symmetric difference $\mb
Z_{-}\Delta S$ is a finite set. The Schubert cells $\Sigma_{S}$ are
in one-to-one correspondence with the $\mb T$-fixed points
$\hs_{S}$, where $\hs_{S}$ is the closed subspace of $\hs$ spanned
by $\{z^{s}:s\in S\}$ (the fixed point $\hs_{S}$ is contained in the
Schubert cell $\Sigma_{S}$). Instead of a subset $S$ of $\mb Z$, we
can consider a decreasing sequence $(s_{n})_{n}$ of integers. It is
easy to check that $s_{n}=-n+d$ for $n\gg 0$, where $d$ is the index
of $\hs_{S}$. The complex codimension of the Schubert cell
$\Sigma_{S}$ is given by the formula
\begin{equation*}
codim \Sigma_S=\sum_{i=1}^{\infty}(s_{i}+i-d).
\end{equation*}
The closure $\bar{\Sigma}_{S}$ of $\Sigma_{S}$ is called the
Schubert cycle of characteristic sequence $S$. It defines a
cohomology class in $H^*(\Gr(\hs))$ having dimension equal to $2codim \Sigma_S$. Since the Schubert cycle
$\bar{\Sigma}_{S}$ is $\mb T$-invariant, it specifies also an element
$\Omega_{S}^{T}$ in $H_{\mb T}^{*}(\Gr(\hs))$. Denote
$\lambda_{n}=s_{n}+n-d$ for $n\geq 1$. Then $(\lambda_{n})$ form a
partition. Instead of using the sequence $S$ to label the
(equivariant) cohomology class $\Omega_{S}^{T}$, we use the notation
$\Omega_{\lambda}^{T}$. Then the dimension of
$\Omega_{\lambda}^{T}$ is equal to $2|\lambda|$. Similarly, we
denote $\Sigma_{S}$ by $\Sigma_{\lambda}$. For more details  see
\cite{LS}, \cite{PS}.

Let us consider the submanifold $\Gr_{d}^{l}$ consisting of points
$W$ in $\Gr_{d}(\hs)$ such that the orthogonal projection
$\pi_{l}:W\to z^{-l}\hs_{-}$ is surjective. There is an equivariant
vector bundle of rank $n=d+l$ over $\Gr_{d}^{l}$ whose fiber over
$W$ is the kernel of the projection $\pi_{l}:W\to z^{-1}\hs_{-}$.
The action of $\mb T$ (or any action of $S^{1}$) on $\Gr_{d}(\hs)$
induces an action on $\Gr_{d}^{l}$. The equivariant Schubert cycle
$\bar{\Sigma}_{\lambda,l}$ in $\Gr_{d}^{l}$ is the intersection of
the equivariant Schubert cycle $\bar{\Sigma}_{\lambda}$ in
$\Gr_{d}(\hs)$ and $\Gr_{d}^{l}$. The dual equivariant cohomology
class of $\bar{\Sigma}_{\lambda,l}$ in $H_{\mb T}^{*}(\Gr_{d}^{l})$
is denoted by $\Omega_{\lambda,l}^{T}$. If $l(\lambda)<d+l$ where
$l(\lambda)$ means the length of the partition  $\lambda$ then the
Schubert cycle  $\bar{\Sigma}_{\lambda}$ and $\Gr_{d}^{l}$ are in
general position; hence we have
\begin{equation*}
f_{l}^{*}\Omega_{\lambda}^{T}=\Omega_{\lambda,l}^{T},
\end{equation*}
where $f_{l}^{*}:H_{\mb T}^{*}(\Gr_{d}(\hs))\to H_{\mb
T}^{*}(\Gr_{d}^{l})$ is the homomorphism induced by the inclusion map
$f_{l}:\Gr_{d}^{l}\to \Gr_{d}(\hs)$.

\begin{prop}
The equivariant Schubert class $\Omega_{\lambda,l}^{T}$ in
$\Gr_{d}(\hs)$ is given by
\begin{equation}\label{schubert}
\Omega_{\lambda,l}^{T}=\det\left[c_{\lambda_{i}+j-i}^{T}(\underline{\hs}_{-l,\lambda_{i}-i+d-1}-\mc
E_{l})\right]_{i,j=1}^{n}.
\end{equation}
Here $\underline{\hs}_{i,j}$ is the equivariant vector bundle
$\hs_{i,j}\times \Gr_{d}^{l}$ and the vector space $\hs_{i,j}$
is the subspace of $\hs$ spanned by $\{z^{k}:i\leq k\leq j\}$.
\end{prop}
\begin{proof}
This statement follows from the Kempf-Laksov's formula, see
\cite{Ful} and \cite {KL}.
\end{proof}

\section{Structure Constants of Schubert Classes with respect to the Standard $S^{1}$-Action}
Let $\hs_{S}$ be a fixed point of $\mb T$--action on
$\Gr_{d}(\hs)$. The inclusion map $\iota_{S}:\{\hs_{S}\}\to
\Gr_{d}(\hs)$ induces a homomorphism:
\begin{equation*}
\iota_{S}^{*}:H_{S^{1}}^{*}(\Gr_{d}(\hs))\to
H_{S^{1}}^{*}(\{\hs_{S}\})
\end{equation*}
called the restriction map. Denote by $\delta=(\delta_{i})$ the
 partition corresponding to $S$, i.e. $\delta_{i}=s_{i}+i-d$. Assume
that $\lambda=(\lambda_{i})$ is a partition such that
$l(\lambda)<l(\delta)$. Then

\begin{lem}\label{fix}
\begin{equation*}
\iota_{S}^{*}\Omega_{\lambda,l}^{T}=u^{|\lambda|}s_{\lambda}^{*}(\delta_{1},\cdots,\delta_{d+l}).
\end{equation*}
Here $s_{\lambda}^{*}$ is the shifted Schur function defined in
\cite{OO}.
\end{lem}
\begin{proof}
Denote $n=d+l$. Assume that $x_{1}',\cdots,x_{n}'$ are equivariant
Chern roots of $\mc E_{l}^{\vee}$ (the dual equivariant vector
bundle of $\mc E_{l}$). Then $c^{T}(\mc
E_{l})=\prod_{i=1}^{n}(1-x_{i}')$. We also have
$c^{T}(\underline{\hs}_{-l,m})=\prod_{j=-l}^{m}(1-ju)$. This gives
us: for each $p\geq 0$
\begin{equation*}
c_{p+j-i}^{T}(\underline{\hs}_{-l,\lambda_{i}-i+d-1}-\mc
E_{l})=\sum_{a+b=p+j-i}h_{a}(x_{1}',\cdots,x_{n}')(-1)^{b}e_{b}(y_{-n},y_{-n-1},\cdots,y_{\lambda_{i}-i-1}),
\end{equation*}
where $y_{j}=(j+d)u$ for $j\in\mb Z$. By the generalized
Jacobi-Trudi formula (\ref{GJT}), we find
\begin{equation*}
\Omega_{\lambda,l}^{T}=\
^{n}s_{\lambda}(x_{1}',\cdots,x_{n}'|\tau^{n+1}y).
\end{equation*}
Applying the restriction homomorphism $\iota_{S}^{*}$, we find
\begin{align*}
\iota_{S}^{*}\Omega_{\lambda,l}^{T}&=\
^{n}s_{\lambda}(\delta_{1}-1+d,\cdots,\delta_{n}-n+d|\tau^{n+1}a)u^{|\lambda|}=\
^{n}s_{\lambda}^{*}(\delta_{1},\cdots,\delta_{n}|a)u^{|\lambda|},
\end{align*}
where $a=(a_{j})$ with $a_{j}=j+d$ for all $j\in\mb Z$. The shifted
double Schur function
$^{n}s_{\lambda}^{*}(\delta_{1},\cdots,\delta_{n}|a)$ coincides with
the shifted Schur function defined in \cite{OO}. This completes the
proof.
\end{proof}
Now we are ready to prove the theorem \ref{coeff0}.  Since $C_{\lambda\mu}^{\nu}(u)$
is of the form $C_{\lambda\mu}^{\nu}u^{|\lambda|+|\mu|-|\nu|}$,
(\ref{coef0}) can be rewritten as
\begin{equation*}
\Omega_{\lambda}^{T}\Omega_{\mu}^{T}=\sum_{\nu}C_{\lambda\mu}^{\nu}u^{|\lambda|+|\mu|-|\nu|}\Omega_{\nu}^{T}.
\end{equation*}
Applying $\iota_{S}^{*}$ to the above equation\footnote{Here  we may
assume that $l(\lambda),l(\mu),l(\nu)<l(\delta)$. If
$l(\lambda)>l(\delta)$, $\iota_{S}^{*}\Omega_{\lambda}^{T}=0$} and
using lemma \ref{fix}, we find that the constants
$(C_{\lambda\mu}^{\nu})$ satisfy (\ref{coef1}).

In order to compute the structure constants for the equivariant
Schubert classes of $\Gr_{d}(\hs)$, we have to introduce the notion
of shifted symmetric functions.

\section{Algebra of Shifted Symmetric Functions}
Let $x=(x_{n})_{n\in\mb N}$ be a sequence of variables obeying
$x_n=0$ for $n>>0$ and $y=(y_{i})_{i\in\mb Z}$ be a doubly infinite
sequence of variables. Denote the set of pairs $(x,y)$ by $R$. Let
us consider a function $f(x|y)$ such that its restriction $f_n$ to
the subset $R_k$  specified by the condition $x_{k+1}=x_{k+2}=\dots
=0$ is a polynomial for every $k\in \mb N.$ We say that $f$ is
shifted symmetric if $f_n$ symmetric with respect to the variables
$x_{i}'=x_{i}+y_{-i}$ for $1\leq i\leq n$. In other words,
\begin{equation*}
f'_n(x_{1}',\dots,x_{n}'|y)=f_n(x_{1}'-y_{-1},\dots,x_{n}'-y_{-n}|y)
\end{equation*}
is symmetric with respect to $x'=(x_{1}',\dots,x_{n}')$.

This definition is motivated by the definition in \cite{OO}. (If we
replace $y_{j}$ by $\mbox{constant}+j$, we obtain the definition of
the shifted symmetric functions given in \cite{OO}.) An essentially
equivalent notion was introduced in \cite {Mol2}.  Instead of  $R$
one can consider a set $\tilde R$ of pairs $(x,y)$ where
$x=(x_{n})_{n\in \mb N}$ and $y=(y_{i})_{i\in\mb Z}$ are sequences
obeying $x_n=y_{-n}$ for $n>>0.$ Shifted symmetric functions on $R$
correspond to symmetric functions on $\tilde R$; this correspondence
can be used to relate our approach to the approach of \cite {Mol2}.
It is obvious that shifted symmetric functions on $R$ constitute a
ring; we denote this ring  by $\Lambda^{*}(x\|y).$  It is clear
that this ring is isomorphic to the ring $\Lambda (x\|y)$ of
symmetric functions on $\tilde R$ considered in \cite {Mol2}.  It
follows immediately from well known results that the shifted double
Schur functions $\{s_{\lambda}^{*}(x|y)\}$ form a linear basis for
$\Lambda^{*}(x\|y)$ considered as $\mb C[y]$-module..

Now we are ready to finish the proof of theorem \ref{coeff}. Let
$\{x_{1}',\cdots, x_{n}'\}$ be the equivariant Chern roots of of the
equivariant vector bundle $\mc E_{l}^{\vee}$. Denote
$\mathbf{u}=(u_{i})_{i\in\mb Z}$ the sequence of weights of the
action of $\mb T$ on $\hs$ and $y=\tau^{-d}\mathbf{u}$. Define
$x_{i}=x_{i}'-y_{-i}$ for $1\leq i\leq n$.


\begin{thm} The equivariant Schubert class
$\Omega_{\lambda,l}^{T}$ is given by
\begin{equation*}
\Omega_{\lambda,l}^{T}=\
^{d+l}s_{\lambda}^{*}(x_{1},\cdots,x_{d+l}|y).
\end{equation*}
\end{thm}
\begin{proof}
The proof is the same as that in lemma \ref{fix}.
\end{proof}
This proposition allows us to define the algebra isomorphism between
the equivariant cohomology ring $H_{\mb T}^{*}(\Gr_{d})$ and the
algebra of shifted symmetric functions $\Lambda^{*}(x\|y)$. The
isomorphism between them is given by
\begin{equation*}
\Omega_{\lambda}^{T}\mapsto s_{\lambda}^{*}(x|y).
\end{equation*}
This implies that the multiplication table in
$H_{T}^{*}(\Gr_{d}(\hs))$ with respect to the basis
$\{\Omega_{\lambda}^{T}\}$ is the same as the multiplication table
of $\{s_{\lambda}^{*}(x|y)\}$ in $\Lambda^{*}(x\|y)$. This completes
the proof of the theorem \ref{coeff}.

\section{Comultiplication}
Our main goal in this section is to prove theorem \ref{com}. Since
the algebra of shifted symmetric functions $\Lambda^{*}(x\|y)$ is
 isomorphic to the algebra of symmetric functions $\Lambda(x\|y)$
 with an isomorphism given by the change of variables, we will use $\Lambda(x\|y)$
for convenience. Without loss of generality, we may assume that all
 indices $d$ are zero, i.e. we consider the product map
\begin{equation*}
\rho:\Gr_{0}(\hs_{even})\times \Gr_{0}(\hs_{odd})\to \Gr_{0}(\hs).
\end{equation*}
The infinite torus acts on $\hs$ and thus acts on $\hs_{even}$ and
$\hs_{odd}$ in a natural way; the map $\rho$ is equivariant. Since
the Grassmannian is equivariantly formal (see \cite {GKM},\cite
{Ful}) we obtain that
\begin{equation*}
H_{\mb T}^{*}(\Gr_{0}(\hs))=H_{\mb T}^{*}(pt)\otimes
H^{*}(\Gr_{0}(\hs))
\end{equation*}
as a $H_{\mb T}^{*}(pt)$-module. Using the K\"{u}nneth theorem and relation
$\mb C[\mathbf{u}]=\mb C[\mathbf{u}_{even}]\otimes \mb
C[\mathbf{u}_{odd}]$, we find that
\begin{equation*}
H_{\mb T}^{*}(\Gr(\hs_{even})\times \Gr(\hs_{odd}))=H_{\mb
T}^{*}(\Gr(\hs_{even}))\otimes H_{\mb T}^{*}(\Gr(\hs_{odd}))
\end{equation*}
as $H_{\mb T}^{*}(pt)$-modules, where
$\mathbf{u}_{even}=(u_{2k})_{k\in\mb Z}$ and
$\mathbf{u}_{odd}=(u_{2k+1})_{k\in\mb Z}$.

For each $l\geq 1$, let $\Gr_{0}^{l}(\hs_{even})$ and
$\Gr_{0}^{l}(\hs_{odd})$ be respectively the submanifolds of
$\Gr_{0}(\hs_{even})$ and $\Gr_{0}(\hs_{odd})$ consisting of points
$W$ so that the orthogonal projection $\pi_{even,l}:W\to
z^{-2[l/2]}\hs_{even,-}$ is surjective if $W\in\Gr_{0}(\hs_{even})$
and $\pi_{odd,l}:W\to z^{-2[l/2]-2}\hs_{odd,-}$ is surjective if
$W\in\Gr_{0}(\hs_{odd})$. We can restrict the equivariant map
(\ref{product}) to the equivariant map
\begin{equation*}
\rho:\Gr_{0}^{l}(\hs_{even})\times \Gr_{0}^{l}(\hs_{odd})\to
\Gr_{0}^{l}.
\end{equation*}
Then the pull back bundle $\rho^{*}\mc E_{l}$ over
$\Gr_{0}^{l}(\hs_{even})\times\Gr_{0}^{l}(\hs_{odd})$ splits into
the direct sum of $\mb T$-equivariant vector bundles $\mc
E_{l,even}$ and $\mc E_{l,odd}$, where the vector bundle $\mc
E_{l,even}$ and $\mc E_{l,odd}$ are respectively the $\mb
T$-equivariant bundle over $\Gr_{0}^{l}(\hs_{even})$ and over
$\Gr_{0}^{l}(\hs_{odd})$) whose fiber over $W$ is the kernel of the
orthogonal projection $\pi_{even,l}:W\to z^{-2[l/2]}\hs_{even,-}$ if
$W\in\Gr_{0}^{l}(\hs_{even})$ and is the kernel of $\pi_{odd,l}:W\to
z^{-2[l/2]-2}\hs_{odd,-}$ if $W\in\Gr_{0}^{l}(\hs_{odd})$.
Similarly, the $\mb T$-equivariant vector bundle
$\rho^{*}\underline{\hs}_{-l,-1}$ over
$\Gr_{0}^{l}(\hs_{even})\times\Gr_{0}^{l}(\hs_{odd})$ splits into a
direct sum of $\mb T$-equivariant bundles $\mc F_{l,even}$ and $\mc
F_{l,odd}$, where $\mc F_{l,even}$ and $\mc F_{l,odd}$ are the
product bundles of the form $F_{even}\times
(\Gr_{0}^{l}(\hs_{even})\times\Gr_{0}^{l}(\hs_{odd}))$ and
$F_{odd}\times
(\Gr_{0}^{l}(\hs_{even})\times\Gr_{0}^{l}(\hs_{odd}))$; $F_{l,even}$
and $F_{l,odd}$ are the linear spaces spanned by $\{z^{2s}:-l\leq
2s\leq -1\}$ and by $\{z^{2s+1}:-l\leq 2s+1\leq -1\}$ respectively.
In other words, we have
\begin{equation*}
\rho^{*}\mc E_{l}=\mc E_{l,even}\oplus \mc
E_{l,odd}\quad\mbox{and}\quad \rho^{*}\underline{\hs}_{-l,-1}=\mc
F_{l,even}\oplus \mc F_{l,odd}
\end{equation*}
as $\mb T$-equivariant vector bundles over
$\Gr_{0}^{l}(\hs_{even})\times\Gr_{0}^{l}(\hs_{odd})$.

Assume that $x_{1},\cdots,x_{l}$ are the equivariant Chern roots of
$\mc E_{l}$ and define the power sum functions
$p_{k}(x_{1},\cdots,x_{l}|\mathbf{u})$, $k\geq 1$ by
\begin{equation}\label{pk}
p_{k}(x_{1},\cdots,x_{l}|\mathbf{u})=\sum_{i=1}^{l}\left[x_{i}^{k}-u_{-i}^{k}\right].
\end{equation}
Then the power sum function $p_{k}$ equals to the $k$-dimensional component of the equivariant Chern
character of the difference bundle $\mc
E_{l}-\underline{\hs}_{-l,-1}$, i.e.
$p_{k}(x_{1},\cdots,x_{l}|\mathbf{u})=\ch_k^{T}(\mc
E_{l}-\underline{\hs}_{-l,-1})$. Since $\ch^{T}(\mc
E_{l}-\underline{\hs}_{-l,-1})=\ch^{T}(\mc E_{l,even}-\mc
F_{l,even})+\ch^{T}(\mc E_{l,odd}-\mc F_{l,even})$, we find that
\begin{equation}\label{comp2}
\rho^{*}p_{k}(x_{1},\cdots,x_{l}|\mathbf{u})=p_{k}(x_{2},x_{4}\cdots,x_{2[l/2]}|\mathbf{u}_{even})+p_{k}(x_{1},x_{3},\cdots,x_{2[(l-1)/2]+1}|\mathbf{u}_{odd}).
\end{equation}
Notice that the function (\ref {pk}) can be extended to an element
of the ring $\Lambda (x\|\mathbf{u})$ ( to a symmetric function on
$\tilde R$). More precisely, the formula
\begin{equation}\label{pkk}
p_{k}(x|\mathbf{u})=\sum_{i=1}^{\infty}\left[x_{i}^{k}-u_{-i}^{k}\right]
\end{equation}
is a finite sum on $\tilde R$ and therefore specifies an element of
$\Lambda (x\|\mathbf{u})$. This formula is also equivalent to
(\ref{spower}) under a change of variables. By (\ref{comp2}) and
(\ref{pkk}), we obtain the following formula:
\begin{equation}\label{copro}
\rho^{*}p_{k}(x|\mathbf{u})=p_{k}(x_{even}|\mathbf{u}_{even})+p_{k}(x_{odd}|\mathbf{u}_{odd})
\end{equation}
where $(x_{even})_k=x_{2k}$ and $(x_{odd})_k=x_{2k+1}$ for $k\in\mb
Z$.

Let $\psi_{even}$ and $\psi_{odd}$ be the algebra isomorphisms
$\Lambda(x_{even}\|\mathbf{u}_{even})\to \Lambda(x\|\mathbf{u})$ and
$\Lambda(x_{odd}\|\mathbf{u}_{odd})\to \Lambda(x\|\mathbf{u})$
defined by $\psi_{even}(x_{2k})=x_{k}$, $\psi_{even}(u_{2k})=u_{k}$
and by $\psi_{odd}(x_{2k+1})=x_{k}$, $\psi_{even}(u_{2k+1})=u_{k}$
for $k\in\mb Z$ respectively. Then $\psi_{even}\otimes\psi_{odd}$
determines an isomorphism from $H_{\mb
T}^{*}(\Gr_{0}(\hs_{even}))\otimes H_{\mb
T}^{*}(\Gr_{0}(\hs_{odd}))$ to $H_{\mb T}^{*}(\Gr_{0}(\hs))\otimes
H_{\mb T}^{*}(\Gr_{0}(\hs))$.  Combining the maps $\rho^{*}$and
$\psi_{even}\otimes\psi_{odd}$ we obtain a homomorphism
\begin{equation*}
\Delta:H_{\mb T}^{*}(\Gr_{0}(\hs))\to H_{\mb
T}^{*}(\Gr_{0}(\hs))\otimes H_{\mb T}^{*}(\Gr_{0}(\hs))
\end{equation*}
so that (\ref{comp}) holds by (\ref{copro}). This completes the
proof of the theorem \ref{com}.\\

{\bf Acknowledments} Both authors thank Max-Planck Institute
f\"{u}r Mathematik in Bonn for generous support and wonderful
environment. The second author was partially supported by the NSF
grant DMS-0805989.

\bibliographystyle{amsplain}

\end{document}